\documentclass[11pt,a4paper]{amsart}
\usepackage[english]{babel}
\usepackage{amsmath,amsthm, amssymb, amsfonts}
\usepackage{setspace}
\usepackage{anysize}
\usepackage{url}
\usepackage{bm}
\usepackage{color,graphicx}
\usepackage[colorlinks=false]{hyperref}
\usepackage{paralist}
\usepackage{verbatim}
\usepackage[utf8]{inputenc}

\theoremstyle{plain}
\newtheorem{theorem}{\bf Theorem}[section]

\newtheorem*{theorem*}{Theorem}
\newtheorem{proposition}[theorem]{\bf Proposition}
\newtheorem{lemma}[theorem]{\bf Lemma}
\newtheorem{corollary}[theorem]{\bf Corollary}
\newtheorem{conjecture}[theorem]{\bf Conjecture}
\newtheorem*{conjecture*}{\bf Conjecture}
\theoremstyle{definition}
\newtheorem{definition}[theorem]{\bf Def\mbox{}inition}
\newtheorem{remark}[theorem]{\bf Remark}
\newtheorem{notation}[theorem]{\bf Notation}
\theoremstyle{remark}
\newtheorem{example}[theorem]{\bf Example}
\theoremstyle{example}

\def\subclassname{{\bfseries Mathematics Subject Classification
(2010)}\enspace}
\def\subclass#1{\par\addvspace\medskipamount{\rightskip=0pt plus1cm
\def\and{\ifhmode\unskip\nobreak\fi\ $\cdot$
}\noindent\subclassname\ignorespaces#1\par}}

\def \HF{{\operatorname{HF}}}
\def \HS{{\operatorname{HS}}}
\def \HV{{\operatorname{h}}}

\def \type{{\operatorname{type}}}
\def \Tor{{\operatorname{Tor}}}

\def \link{{\operatorname{link}}}
\def \lk{{\operatorname{link}}}
\def \lk{{\operatorname{link}}}

\def \G{{\Gamma}}
\def \NN{\mathbb N}
\def \M{\mathcal M}
\def \uno{{\bf 1}}
\def \aaa{{\bf a}}
\def \bbb{{\bf b}}

\def \ZZ{\mathbb Z}
\def \D{\Delta}

\def \DM{\Delta_{\operatorname{max}}}
\def \Dm{\Delta_{\operatorname{min}}}
\def \Ds{^{\operatorname{si\!}}\Delta}
\def \Dp{\Delta^{\prime}} 

\def \F{\mathcal F}
\def \FD{{\mathcal F}(\Delta)}

\def \P{\mathcal P}

\def\AL{\ensuremath{c}}

\begin{document}

\title{$h$-vectors of matroid complexes}
\author{Alexandru Constantinescu}
\address{Institut de Mat\'ematiques, Universit\'e de Neuch\^atel,
 Emile-Argand 1, 
  Neuch\^atel, Switzerland}
\email{\url{alexandru.constantinescu@unine.ch}}
\urladdr{\href{http://math.unibas.ch/institut/personen/profil/profil/person/constantinescu/}{http://math.unibas.ch/institut/personen/profil/profil/person/constantinescu/}} 
\author{ Matteo Varbaro }
\address{ Dipartimento di Matematica,
 Universit\`a  di Genova, Via Dodecaneso 35, Genova 16146, Italy}
\email{\url{varbaro@dima.unige.it}}
\urladdr{\href{http://www.dima.unige.it/~varbaro/}{www.dima.unige.it/~varbaro/}} 
\date{{\small \today}}


\maketitle
\begin{abstract}
\noindent 
In this paper we partition in classes the set of matroids of fixed dimension on a fixed vertex set. In each class we identify two special matroids, respectively with minimal and maximal $h$-vector in that class. Such extremal matroids also satisfy a long-standing conjecture of Stanley. As a byproduct of this theory we establish Stanley's conjecture in various cases, for example the case of Cohen-Macaulay type less than or equal to $3$.
\end{abstract}

\section*{Introduction}
In 1977 Stanley conjectured that the $h$-vectors of matroids are pure {\it O}-sequences \cite[p.59]{St1},  that is they are $h$-vectors of Artinian monomial level algebras or, equivalently, $f$-vectors of pure order ideals.
Ever since, the $h$-vectors of matroids have been in  the focus of many researchers (see \cite{Hi1,Hi2,Ch, MSSVWZ97,Sw04,Sp09}). 
Pure {\it O}-sequences themselves have attracted a lot of attention as well,  quite a few conjectures being made regarding their shape (\cite{BMMNZ} gives an overview of the topic). 
Although several researchers have approached Stanley's conjecture, to our knowledge  only very specific cases  have been proven. The case of cographic matroids was proven 
in \cite{Ch,Me}, that of lattice path matroids in \cite{Sc} and more generally the one of  cotransversal matroids in \cite{Oh12}.
Low rank and degree situations were recently investigated in  \cite{Sto, Sto09, HSZ}.

In the present paper we prove Stanley's conjecture in several cases, which appear in every rank and codimension. As a particular case, we obtain the conjecture for  all matroid complexes of Cohen-Macaulay type less than or equal to $3$. 
For any positive integers $n$ and $d$, we divide  the $(d-1)$-dimensional matroids on $n$ vertices  in different classes, which are   indexed by the partitions of $n$ with length at least $d$. 
For each class we build the set of all possible $h$-vectors of the duals of the matroids in the respective class. We then 
 identify two special matroids whose duals have minimal, respectively  maximal $h$-vectors in that set.  For all these extremal matroids  we prove in  a constructive way that Stanley's conjecture holds. \\

Our approach passes via an equivalent phrasing of Stanley's conjecture. The $h$-vector of a matroid $\D$ is defined as  the $h$-vector of the corresponding Stanley-Reisner ring and we will denote it by $h_\D$. To a simplicial complex in general, apart from the Stanley-Reisner ideal $I_\D$, one can associate its vertex cover ideal $J(\D)$. 
We will denote the $h$-vector of the quotient ring of $J(\D)$ by $h^{\D}$.
If we denote by $\D^c$ the dual of $\D$ (that is the simplicial complex generated by the complements of the facets in the vertex set), we have that 
\[ J(\D^c) = I_{\D}\quad\quad\textrm{and}\quad\quad  h^{\D^c}=h_\D.\]
A classical theorem of matroid theory says that $\D$ is a matroid if and only if $\D^c$ is a matroid. This implies the following equivalent formulation of  Stanley's conjecture:
\begin{conjecture*}[Stanley]
For any matroid $\D$, the vector $\HV^{\D}$ is a pure $O$-sequence.
\end{conjecture*}

Let us summarize  the contents of the paper. Section \ref{prelsec} is mainly devoted to preliminary results and  establishing the notation. Nevertheless, we show in Corollary \ref{linkage} an equality involving the cover ideals of certain matroid complexes. This equality supplies the exact sequence \eqref{H}, which will be a crucial tool throughout the paper. 
The existence of this exact sequence depends heavily on the properties of  matroids. 
 In Remark \ref{icp} we also present a counterexample to the Interval Conjecture for Pure {\it O}-sequences formulated  by Boij et al. in \cite{BMMNZ}. 

In Section \ref{sec2},  we first provide some structural results for matroid complexes. We  show that the 1-skeleton of a matroid is a complete $p$-partite graph. 
The division of the matroids into classes will be done in correspondence with these partitions of the vertex set. In each class we then define $d-1$ matroids: $\D_t(d,p,\aaa)$, for $t=0,\ldots ,d-2$, where $\aaa$ is the partition of $n$. All these matroids are representable over fields with ``enough'' elements, and in most cases they are neither graphic nor transversal. 
We will call  $\D_0(d,p,\aaa)$  complete $p$-partite matroids. These  are a simultaneous generalization of both uniform and partition matroids.

Later on in this section, we attach to each matroid $\D$ another matroid $\Ds$, named simplified matroid, of the same dimension but on less vertices. The simplified matroid reflects many properties of the original matroid. For example, the total Betti numbers of $J(\D)$ and $J(\Ds)$ are the same (Proposition \ref{betti}). In Proposition \ref{characterizationcodimension2} we provide  a formula which computes $\HV^{\D}$ for $1$-dimensional matroids. It turns out that the set of $h$-vectors of matroid complexes of the type $(1,2,h_2,\ldots ,h_s)$ coincides with the set of pure $O$-sequences of the form $(1,2,h_2,\ldots ,h_s)$.

In Section \ref{sec3} we prove the conjecture of Stanley in various instances. In  Theorem \ref{stanleyforcomplete} we show that $\HV^{\D}$ is a pure $O$-sequence whenever $\D$ is a $(d-1)$-dimensional complete $p$-partite matroid for some $p\geq d$. 
Using Theorem \ref{stanleyforcomplete}, we prove the more general statement that  $\HV^{\D_t(d,p,\aaa)}$ is a pure $O$-sequence for all $t = 0,\ldots,d-2$ (Theorem \ref{thmdeltat}).

In Section \ref{sec4}, for any partition $\aaa$ of $n$ with $p\geq d$ parts, we denote by $\M(d,p,\aaa)$ the set of $(d-1)$-matroids on $n$ vertices, whose $1$-skeleton is $p$-partite and the cardinalities of the partition sets correspond to $\aaa$. By the results  of Section \ref{sec2}, every matroid belongs to exactly one of these sets. 
In Theorems \ref{minimality} and \ref{maximality}, we show that
\[\HV^{\D_{d-2}(d,p,\aaa)}\leq \HV^{\D}\leq \HV^{\D_0(d,p,\aaa)}, \ \ \forall \ \D\in \M(d,p,\aaa).\]
A priori the existence of a matroid in $\M(d,p,\aaa)$ with minimal $h$-vector is not clear at all. Indeed, a striking consequence is the validity of Stanley's conjecture whenever the Cohen-Macaulay type of $S/I_{\D}$ is less than or equal to three. In other words, we establish Stanley conjecture for all the $h$-vectors of type $(h_0,h_1,\ldots ,h_s)$ with $h_s\leq 3$. 

\vspace{2mm}

\subsection*{Acknowledgement} We would like to thank Thomas Kahle for helpful discussions that led to improvements of the paper.

\section{Preliminaries}\label{prelsec}

In this section we will recall most of the algebraic and combinatorial  notions that we will use throughout the paper. For general aspects on the topics presented below we refer the reader to the books of Stanley \cite{St}, of Bruns and Herzog \cite{BH} and of Oxley \cite{Ox11}.

For a positive integer $n$ denote by $[n]$ the set $\{1,\ldots,n\}$. 
A simplicial complex $\Delta$ on $[n]$ is a collection of subsets of $[n]$ such that $ F \in \Delta$ and $ F ' \subset  F$ imply $F ' \in \Delta$. 
Notice that we are not requiring that $\bigcup_{F\in\D}F=[n]$, therefore $\D$ can be viewed as a simplicial complex on any overset of $\bigcup_{F\in\D}F$. Each element $F\in\D$ is called a \emph{face} of $\D$. 
The dimension of a face $F$ is $|F| -1$ and the \emph{dimension of $\D$} is $\max\{\dim F ~:~ F \in \D\}$. 
A maximal face of $\D$ with respect to inclusion is called a $\emph{facet}$ and  we will denote by $\mathcal{F}(\Delta)$ the set of facets of $\Delta$.  
A simplicial complex is called \emph{pure} if all facets have the same cardinality.
We call a vertex $v$ a \emph{cone point} of $\D$ if $v \in F$ for any $F \in \mathcal{F}(\D)$. If $F_1,\ldots ,F_m$ are subsets of $[n]$, then we denote by $\langle F_1,\ldots ,F_m\rangle$ the smallest simplicial complex on $[n]$ containing them. Explicitly:
\[\langle F_1,\ldots ,F_m\rangle =\{F\subset [n] \ : \ \exists \ i \in\{1,\ldots ,m\}:F\subset F_i\}.\]
We say that $F_1,\ldots ,F_m$ generate the simplicial complex $\langle F_1,\ldots ,F_m\rangle$. Clearly every simplicial complex is generated by its set of facets. For any face $F$  the \emph{link} of $F$ in $\D$ is the following simplicial complex: 
\[\link_{\D}F = \{ F' \in \D ~:~ F' \cup F \in \D \textrm{~and~} F' \cap F = \emptyset\}.\]
For a set of vertices $W\subset [n]$, the \emph{restriction} of $\D$ to $W$ is the following subcomplex of $\D$:
\[ \D|_W = \{F\in\D~:~F\subset W \}.\]
The subcomplex $\D|_W$ is also  called the subcomplex of $\D$ \emph{induced by} the vertex set $W$.
If $F$ is a face of $\D$, then the  \emph{face deletion} of $F$ in $\D$ is $\D \setminus F = \{F' \in\D~:~F\nsubseteq F'\}$. Whenever $F$ is a 0-dimensional face $\{v\}$ we will just write $\D\setminus v$ for the face deletion of $\{v\}$ and $\lk_\D v$ for the link of $\{v\}$. Notice that $\D\setminus v =\D|_{[n]\setminus \{v\}}$ for all $v\in[n]$.
The \emph{dual complex} of $\D$ is the simplicial complex $\D^c$ on $[n]$ with facets
\[ \F(\D^c) = \{[n]\setminus F~:~ F\in \FD\}.\]
For any integer $0 \le k \le \dim\D$, the \emph{$k$-skeleton} of $\D$ is defined as the simplicial complex with facet set $\{F\in\D~:~\dim F = k\}$.

We will now associate to a simplicial complex two square-free monomial ideals. We will then see how these ideals are related via the dual complex. 
Denote by $S=\Bbbk[x_1,\ldots,x_n]$ the polynomial ring in $n$ variables over a field $\Bbbk$. For each subset $F\subset[n]$ define the monomial $\texttt{x}_{F}$ and the prime ideal $\P_F$ as follows:
\begin{eqnarray*}
\texttt{x}_{F} &=& \prod_{i\in F} x_i, \\
\P_F &=& (x_i ~:~ i\in F).
\end{eqnarray*}
The \emph{Stanley-Reisner ideal} of $\D$ is the ideal $I_{\D}$ of $S$  generated by the square-free monomials $\texttt{x}_{F},$ with $F \notin\D$. In particular we have
\[I_\D = (\texttt{x}_{F}~:~F \textrm{~is a minimal nonface of~}\D).\]
The second square-free monomial ideal we can associate to $\D$ is the \emph{cover ideal} of $\D$, namely
\[J(\D) = \bigcap_{F\in\FD}\P_F.\]
The name \lq\lq cover ideal\rq\rq~ comes from the following fact. A collection of vertices $A\subset [n]$ is called a \emph{vertex cover} of $\D$ if $A\cap F \neq \emptyset$ for any $F\in\FD$. A vertex cover $A$ is called \emph{basic} if no proper subset of $A$ is again a vertex cover. It is easy to check that we have
\[J(\D) = (\texttt{x}_A~:~ A \textup{~ is a basic vertex cover of~}\D).\]
It is a well known fact that the prime decomposition of the Stanley-Reisner ideal is $$I_\D = \bigcap_{F\in\FD} \P_{[n]\setminus F}.$$ 
The following equality, which follows directly from the definition, will be very important for the approach of this paper:
  $$J(\D)=I_{\D^c}.$$

We denote by $\Bbbk[\D] = S/I_\D$ the \emph{Stanley-Reisner ring} of $\D$. Let $h_{\Bbbk[\D]}=(h_0,h_1,\ldots,h_s)$ be  its $h$-vector.  If $\HS_{\Bbbk[\D]}(t)$ is the Hilbert series of $\Bbbk[\D]$, then we have
\[ \HS_{\Bbbk[\D]}(t) = \frac{h_0+h_1t + \ldots + h_st^s}{(1-t)^d},\]
where  $h_s\neq 0$ and   $ d = \dim \Bbbk[\D] = \dim\D +1.$

In the classical terminology, the $h$-vector of a simplicial complex is the $h$-vector of its Stanley-Reisner ring. 
As we will mainly deal with cover ideals, in order to avoid the over-use of the word dual, we will fix the following notation and terminology.

\begin{notation} Let $\D$ be any simplicial complex.
\begin{itemize}
\item[1.] We denote the $h$-vector of $\Bbbk[\D]$ by $h_\D$.
\item[2.] We denote the $h$-vector of $S/J(\D)$ by $h^\D$.
\item[3.] We will refer throughout this paper to $h^\D$ as \emph{the $h$-vector of $\D$}.
\end{itemize}
\end{notation}
\noindent Notice that we have the correspondence $h^\D$ = $h_{\D^c}$.

For a $(d-1)$-dimensional simplicial complex $\D$ on $[n]$ such that $S/J(\D)$ is Cohen-Macaulay, we denote by $\type(\D)$ the last total Betti number in the minimal free resolution of $S/J(\D)$, namely 
$$\type(\D)=\beta_d(S/J(\D))=\dim_{\Bbbk}\Tor_d^S(S/J(\D),\Bbbk).$$

Matroid theory was born out of the need to study the concept of dependence in an abstract way. In this paper we will view  matroids as simplicial complexes whose faces correspond to the independent sets. A characteristic of matroids is that they admit many different but equivalent definitions (see \cite{Ox11} and \cite[Chapter III.3]{St}). 
We present here three of them. 
\begin{definition}
A simplicial complex $\D$ is called a \emph{matroid complex} (or just matroid) if one of the following equivalent properties hold:
\begin{itemize}
\item[1.] \emph{The augmentation axiom:} For any two faces $F, G \in \D$ with $|F| < |G|$ there exists $i \in G$ such that $F \cup \{i\} \in \D$.
\item[2.] \emph{The exchange property:} For any two facets $F, G \in \FD$ and for any $i \in F$ there exists a $j \in G$ such that $(F\setminus \{i\}) \cup \{j\} \in \D$.
\item[3.] For any subset $W\subset [n]$ the restriction $\D|_W$ is pure.
\end{itemize}
\end{definition}
\noindent A basic result in matroid theory that we will exploit a lot is the following:
\begin{theorem}[\cite{Ox11}{Theorem 2.1.1}]\label{duality}
A simplicial complex $\D$ on $[n]$ is a matroid if and only if $\D^c$ is a matroid.
\end{theorem}

An algebraic characterization of matroid complexes has been given in \cite{MT} and \cite{Va}, namely a simplicial complex $\D$ is a matroid if\mbox{f} all the symbolic powers of $I_\D$  are Cohen-Macaulay. 
Another algebraic property that will be important for us (even if it does not characterize matroids) is the following:
the Stanley-Reisner ring of a matroid is level (\cite[Chapter III, Theorem 3.4]{St}). This means  that $\Bbbk[\D]$ is Cohen-Macaulay and the socle of its Artinian reduction lies in exactly one degree. A prototype of level algebras are the Gorenstein algebras, which correspond to socle dimension $1$.
An important consequence of  $S/J(\D)$ being level is that the type can be expressed only in terms of the last entry of the $h$-vector, namely  $$\type(\D)=\HV^{\D}(s) \mbox{ \ \ \ where \ }s=\max\{i~:~\HV^{\D}(i)\neq 0\}.$$

The following lemma is important because it provides  a recursive formula for the $h$-vectors. This formula will be a main  ingredient in many of our proofs. The lemma itself can also be interpreted from a liaison-theoretical point of view, namely it is easy to check that the ideal relation provides a basic double link. 

\begin{lemma}\label{linkage}
If $\D$ is a matroid on $[n]$ and  $v\in\D$ is a vertex that is not a cone point, then
\[J(\D)=x_vJ(\D\setminus v)+J(\lk_\D v).\]
\end{lemma}
\begin{proof} We will first make a general observation. Let $\Gamma$ be a simplicial complex on $[n]$ and  consider $\Gamma\setminus n$ and $\lk_{\Gamma}n$ as simplicial complexes on $[n-1]$. 
It is straightforward to show that, if  $\Gamma\setminus n$ is pure and  of the same dimension as $\Gamma$, then $(\Gamma\setminus n)^c=\lk_{\Gamma^c} n$.

We may obviously assume that $v = n$. 
Since $\D$ is a matroid, $\D\setminus n$ is pure. Since $n$ is not a cone point,   it has the same dimension as $\D$.  Therefore, by the above observation we have
\[(\D\setminus n)^c=\lk_{\D^c} n.\]
As we assumed that $n\in\D$, we have that $n$ is not a cone point of $\D^c$. By Theorem \ref{duality} and the general observation we obtain also that
\[\D^c\setminus n = ((\D^c\setminus n)^c)^c = (\lk_{(\D^c)^c} n)^c=(\lk_{\D} n)^c.\]

For all simplicial complexes $\Gamma$ on $[n]$ we have the following  equality for Stanley-Reisner ideals:
\[I_{\Gamma}=x_nI_{\lk_{\Gamma}n}+I_{\Gamma\setminus n}.\]
Exploiting it for $\Gamma=\D^c$,  together with $I_{\D^c}=J(\D)$ and  all the above observations, we conclude.
\end{proof}
\begin{remark}Whereas the equality $I_{\D}=x_nI_{\lk_{\D}n}+I_{\D\setminus n}$ holds true for any simplicial complex, the equality $J(\D)=x_vJ(\D\setminus v)+J(\lk_\D v)$ depends strongly on the fact that $\D$ is a matroid. 
For instance,  Lemma \ref{linkage}  already fails for any vertex of a path of length three.
\end{remark}
\begin{remark}
From the point of view of Gorenstein liaison, Lemma \ref{linkage}  implies that the ideals $J(\D)$ can be linked to a complete intersection. A more general statement in this direction has been proven  by Nagel and R\"omer in \cite{NaRo}. 
\end{remark}

\begin{remark}
Lemma \ref{linkage} gives rise to the exact sequence
\[0\longrightarrow J(\lk_{\D}v)(-1)\longrightarrow J(\D\setminus v)(-1)\oplus J(\lk_{\D}v) \longrightarrow J(\D) \longrightarrow 0.\]
The above exact sequence yields the following relation for the Hilbert functions:
\[\HF_{J(\D)}(k)=\HF_{J(\D\setminus v)}(k-1)+\HF_{J(\lk_{\D}v)}(k)-\HF_{J(\lk_{\D}v)}(k-1) \ \ \ \forall \ k\in \ZZ,\]
which in turn yields     
\[\HF_{S/J(\D)}(k)=\HF_{S/J(\D\setminus v)}(k-1)+\HF_{S/J(\lk_{\D}v)}(k)-\HF_{S/J(\lk_{\D}v)}(k-1) \ \ \ \forall \ k\in \ZZ.\]
Eventually, taking differences, for every matroid $\D$ and every $v\in \D$ that is not a cone point we obtain
\begin{equation}\label{H}
\HV^\D(k)=\HV^{\D\setminus v}(k-1)+\HV^{\lk_{\D}v}(k)\ \ \ \forall \ k\in \ZZ.
\end{equation}
Formula \eqref{H} will be crucial throughout the paper.
\end{remark}

An \emph{order ideal} is a finite collection $\G$ of monomials of some standard graded polynomial ring, such that $M\in\G$ and $N$ divides $M$ imply $N\in\G$. The partial order given by the divisibility of monomials gives $\G$ a poset structure. An order ideal is called \emph{pure} if all maximal monomials have the same degree. To every order ideal $\G$ we   associate its $f$-vector $f(\G) = (f_0(\G),\ldots, f_s(\G))$, where for every $i=0,\ldots,d$ we have
\[f_i(\G)= |\{ M \in \G~:~ \deg(M)=i\}|.\]
A \emph{pure $O$-sequence} is a vector $h=(h_0,\ldots,h_s)$ that can be obtained as the $f$-vector of some pure order ideal.

\begin{remark}
Pure $O$-sequences can also be presented as the  $h$-vectors of Artinian monomial level algebras, i.e. Artinian level algebras $A$ which are isomorphic to $R/I$ for some polynomial ring $R$ and some monomial ideal $I\subset R$. 
It is very easy to see that, in this situation, if $A$ is Gorenstein then $I$ is forced to be a complete intersection. So the pure $O$-sequences of type $(h_0,h_1,\ldots ,h_{s-1},1)$ are well understood: they are $h$-vectors of complete intersections. In particular,  it emerges that pure $O$-sequences are much more special than $h$-vectors of level algebras in general. 

However, already a characterization of pure $O$-sequences of the type $(h_0,h_1,\ldots ,h_{s-1},2)$, i.e. when the Artinian monomial level algebra $A$ has  Cohen-Macaulay type $2$, is not known (see  \cite{BMMNZ}).
\end{remark}

In \cite{St1} Stanley phrased his conjecture in terms of the $h$-vector of the Stanley-Reisner ring. By Theorem \ref{duality}  an equivalent statement is the following:
\begin{conjecture}[Stanley]\label{conjd}
If $\D$ is a matroid, then the $h$-vector of $S/J(\D)$ is a pure $O$-sequence.
\end{conjecture}
Conjecture \ref{conjd} is known for some families; we list here the most general of them. 
\begin{itemize}
\item[1.] When $S/J(\D)$ is Gorenstein, see \cite[Theorem 4.4.10]{Sto}. 
\item[2.] When $\HV^{\D}=(1,h_1,h_2,h_3)$, see \cite{Sto} and \cite{HSZ}. 
\item[3.] When $\D$ is a graphic matroid, see \cite{Me}.
\item[4.] When $\D$ is a transversal matroid, see \cite{Oh12}.
\item[5.] When the dual of $\D$ is a paving matroid, see \cite{MNRV}. This corresponds to the case in which $h_i=\displaystyle \binom{v+i-1}{i}$ for all $i<s$, where $v$ is the number of $0$-dimensional faces of $\D^c$.
\item[6.] When $\HV^{\D}=(1,2,h_2,\ldots ,h_s)$. Indeed, one can see by the Hilbert-Burch theorem that, in the height $2$ case, pure $O$-sequences coincide with $h$-vectors of level algebras (see \cite[Proposition 4.5]{BMMNZ} for the precise proof), so one can deduce the validity of the conjecture in this case by \cite[Chapter III, Theorem 3.4]{St}.
\end{itemize}

Computational experiments  using the computer algebra system CoCoa \cite{cocoa} were an important part in the preparation of this work.
In our investigation, we  found a counterexample to the Interval  Conjecture for Pure {\it O}-sequences (see \cite{BMMNZ}):
\begin{remark}\label{icp}
One can check that the vectors $(1, 4, 10, 13, 12, 9, 3)$ and $(1, 4, 10, 13, 14, 9, 3)$ are pure $O$-sequences. Indeed, the order ideals are generated by $\{x^3y^2z, \ x^3yt^s2, \ x^3z^2t\}$,  respectively by $\{x^4y^2, \ x^3yzt, \ x^2z^2t^2\}$. Looking at all possible choices of three monomials of degree $6$ in $4$ variables, 
it is possible to compute all the pure $O$-sequences of the form $(1,4,h_2,\ldots ,$ $h_5,3)$. 
Checking the obtained list,  one can realize that $(1, 4, 10, 13, 13, 9, 3)$ does not appear among the pure $O$-sequences, a contradiction to the above-mentioned conjecture.
\end{remark}

\section{The structure of matroids}\label{sec2}

In  this paper we will stratify the set of matroids of fixed dimension and on a fixed vertex set in terms of partitions of the vertex set. To this aim, in this section we will prove some technical facts. 
Most of these are  well known facts for  matroid theory specialists, however we consider it convenient to provide  proofs as well.
We will then present  the simplified  matroid associated to any given matroid. This matroid has only trivial parallel classes, but  important information, such as the total Betti numbers $\beta_i(S/J(\D))$, is preserved. 
We conclude the section presenting a formula that computes the $h$-vector of a codimension two Stanley-Reisner ring of a matroid. \\

From now on, unless otherwise stated, we will consider simplicial complexes $\D$ on $[n]$ with the property that $v\in\D$ for all $v \in[n]$. Notice that  the number of vertices not belonging to $\D$ does not influence $h^\D$, so this is no restriction in terms of our goals. This assumption can be also expressed as $[n]=\bigcup_{F\in\D}F$ and  if $\D$ is a $(d-1)$-dimensional matroid on $[n]$, a remark in \cite[p. 94]{St} implies that
\[n-d=\max\{i~:~\HV^{\D}(i)\neq 0\}.\]

The following easy remark is the starting point  for many of  the following technical results.
\begin{remark}\label{matr1}
If $\D$ is a 1-dimensional simplicial complex on $[n]$, then $\D$ is a matroid if and only if for any $v,w \in [n]$ with $\{v,w\} \notin \D$ we have that $\link_{\D}(v)=\link_{\D}(w)$.
\end{remark}
One-dimensional simplicial complexes can be viewed as graphs on the same vertex set; the edges are the faces of dimension one. For this reason we will switch between graph and simplicial complex whenever we find ourselves in this case. Let us recall that a graph is called a \emph{complete $p$-partite} graph if and only if its vertex set can be partitioned into $p$ disjoint nonempty  sets $A_1,\ldots,A_p$ such that $\{v,w\}$ is an edge if and only if $v$ and $w$ lie in different sets of the partition. The following proposition shows that one-dimensional matroids and complete $p$-partite graphs are actually the same thing.  

\begin{proposition}\label{1dim}
If $\D$ is a 1-dimensional matroid, then $\D$ is a complete $p$-partite graph, for some integer $p\ge2$.
\end{proposition}

\begin{proof}
We will proceed by induction on $n$, the number of vertices. Assume that $n\geq 2$, choose $v$ a vertex of $\D$ and consider the set $A_v = \{w \in \D ~:~ \{v,w\}\notin\D\}$. As $\D$ is a matroid, we have by Remark \ref{matr1} that $\lk_{\D}v = \lk_{\D}w$ for any $w\in A_v$. This implies that $A_v$ is an independent set of vertices. Clearly $\lk_{\D}v$ is a $0$-dimensional simplicial complex whose faces correspond to the elements in $[n]\setminus A_v$. Moreover, as one can check by definition, the restriction $\D|_{[n]\setminus A_v}$ is also a matroid.

If $\dim \D|_{[n]\setminus A_v} = 1$, we have by induction that $\D|_{[n]\setminus A_v}$ is a complete $p$-partite graph, with $p$-partition of the vertex set $A_1\cup\ldots\cup A_p$. In this case it follows that $\D$ is a complete $(p+1)$-partite graph with partition $[n]=A_v \cup A_1\cup\ldots\cup A_p$.

If $\dim \D|_{[n]\setminus A_v} = 0$ then $[n]\setminus A_v$ is an independent set of vertices, so $\D$ is a complete bipartite graph with bipartition $[n] = A_v \cup ([n]\setminus A_v)$.
\end{proof}
The next corollary  gives a stratification of  the set of all $(d-1)$-dimensional matroids on $[n]$ that will be crucial throughout this work.
Clearly, the $k$-skeleton of a matroid is again a  matroid, so we have the following. 
\begin{corollary}\label{structurematroids}
Let  $\D$ be a simplicial complex. If $\D$ is  a matroid, then there exists a positive integer $p\ge2$ such that  the 1-skeleton of $\D$ is a complete $p$-partite graph.
\end{corollary}

Before showing the next  technical lemmas let us fix more notation. From now on, exploiting Corollary \ref{structurematroids}, $\D$ will be a $(d-1)$-dimensional matroid on $[n]$, with $p$-partition of its $1$-skeleton $A_1,\ldots,A_p$. We will call the sets of independent vertices given by the $p$-partition \emph{parallel classes}. Whenever necessary we will denote the vertices of a given parallel class as follows 
$$ A_i = \{v_{i,1},v_{i,2},\ldots,v_{i,a_i}\}.$$ 
For any integer $r \in \{1,\ldots,p\}$ and any indices $1\le i_1<\ldots<i_r\le p$ we denote by $\D_{i_1,\ldots,i_r}$ the restriction of $\D$ to the vertex set $A_{i_1}\cup\ldots \cup A_{i_r}$. We call $\D_{i_1,\ldots,i_r}$ the restriction of $\D$ to the parallel classes  $A_{i_1},\ldots,A_{i_r}$.

\begin{lemma}\label{lemma1} If for $r\le d$ parallel classes $A_{i_1},\ldots,A_{i_r}$, with $1\le i_1 < \ldots <i_r \le p,$ there exist $r$ vertices $v_{i_j}\in A_{i_j}$ such that $\{v_{i_1},\ldots,v_{i_r}\} \in \D$, then for any $r$ vertices $u_{i_j}\in A_{i_j}$ we have that $\{u_{i_1},\ldots,u_{i_r}\} \in \D$. 
\end{lemma} 

\begin{proof} 
We may assume without loss of generality that $i_j = j$, for $j = 1,\ldots,r$.  Choose now $r$ vertices $u_j \in A_j$ and assume that $\{u_1,\ldots,u_r\} \notin \D$. Let $s<r$ be the maximum size of a subset of $\{u_1,\ldots,u_r\}$ that belongs to $\D$. Again we may assume that actually $\{u_1,\ldots,u_s\} \in \D$. 
The simplicial complex $\D_{1,\ldots,r}$ is a matroid. Since $\{v_1,\ldots ,v_r\}\in\D_{1,\ldots,r}$ and the 1-skeleton of $\D_{1,\ldots,r}$ is complete $r$-partite, we have  $\dim \D_{1,\ldots,r} = r-1$. As a matroid is pure, we have that $u_{s+1}$ belongs to some $(r-1)$-dimensional facet $F$ of $\D_{1,\ldots,r}$. Notice that, by the $r$-partition of $\D_{1,\ldots,r}$'s 1-skeleton,  the facet $F$ has to contain exactly one vertex from each parallel class. By the augmentation axiom, we know that there exist $r-s$ vertices $w_1,\ldots,w_{r-s}\in F$ such that $\{u_1,\ldots,u_s, w_1,\ldots,w_{r-s}\} \in \D_{1,\ldots,r}$. As a face cannot contain two vertices from the same parallel class, we obtain that  $w_i = u_{s+1}$ for some $1\le i\le r-s$. In particular $\{u_1,\ldots,u_s,u_{s+1}\} \in \D$, a contradiction to the maximality of $s$. 
\end{proof}
\begin{lemma}\label{lemma2} Let $A_i$ be one of the parallel classes of $\D$ and let $v,w \in A_i$. Then 
\[ \lk_\D v = \lk_\D w.\]
\end{lemma}
\begin{proof}
Choose $\{a_1,\ldots,a_{d-1}\} \in \lk_\D v$. The restriction $\D|_{\{v,w, a_1,\ldots,a_{d-1}\}}$ is a $(d-1)$-dimensional pure complex. As $\{v,w\}\notin\D$ we obtain that $\{a_1,\ldots,a_{d-1},w\} \in \D|_{\{v,w, a_1,\ldots,a_{d-1}\}}$ and thus $\{a_1,\ldots,a_{d-1}\} \in \lk_\D w$.
\end{proof}

Exploiting the results of Lemma \ref{lemma1} and Lemma \ref{lemma2} we will simplify notation in the following way. We will write $A_{i_1}\ldots A_{i_r} \in \D$ if there exist vertices $v_{i_j} \in A_{i_j}$ for all $j = 1,\ldots,r$ such that $\{v_{i_1},\ldots,v_{i_r}\}\in\D$. By Lemma \ref{lemma1} this holds for any choice of $r$ vertices, one in each parallel class. As by Lemma \ref{lemma2} the link of all the vertices in one parallel class is the same, we will denote by $\lk_\D A_i$ the link of some vertex $v\in A_i$. 
These two lemmas lead us  to the following definition (see  \cite[p. 49] {Ox11} for the classical matroid-theoretical definition). 

\begin{definition}
Let $\D$ be a simplicial complex with complete $p$-partite 1-skeleton, satisfying Lemma \ref{lemma1}.  Let $A_1\cup\ldots\cup A_p$ be the $p$-partition and choose for each $i=1,\ldots,p$ a vertex $v_{i,1} \in A_i$.
We define the associated \emph{simplified} complex as\\
\[ \Ds=\D|_{\{v_{1,1},\ldots,v_{p,1}\}}. \]
\end{definition}
We will call a parallel class of $\D$ a \emph{cone class} if  the corresponding vertex in $\Ds$ is a cone point of $\Ds$. This is clearly equivalent to every facet of $\D$ containing a vertex of that parallel class.

\begin{remark} Let $\D$ be a simplicial complex with complete $p$-partite 1-skeleton. Then, using Lemma \ref{lemma1}, we have 
\[ \D \textrm{ is a matroid~} \iff \Ds \textrm{ is a matroid.}\]
\end{remark}

The next proposition shows the close relation between a matroid $\D$ and $\Ds$. 

\begin{proposition}\label{betti}
Given a matroid $\D$ on $[n]$, we have $\beta_i(S/J(\D))=\beta_i(S/J(\Ds))$ for all $i$. In particular, $\type(\D)=\type(\Ds)$.
\end{proposition} 
\begin{proof}
Set $R=\Bbbk[y_1,\ldots ,y_p]$, and consider the $\Bbbk$-algebra homomorphism
$$
\begin{array}{rcl}
\phi: R & \longrightarrow & S\\
 y_i & \mapsto & \prod_{j\in A_i}x_j=m_i
\end{array}
$$
One can check that $\phi(J(\Ds))S=J(\D)$. Moreover it is obvious that $m_1,\ldots ,m_p$ form a regular sequence of $S$, so by a theorem of Hartshorne (\cite[Proposition 1]{Ha2}) $S$ is a flat $R$-module via  $\phi$. So, if $F_{\bullet}$ is a minimal free resolution of $R/J(\Ds)$ over $R$, then it follows that $F_{\bullet}\otimes_{R}S$ is a minimal free resolution of $S/J(\D)$ over $S$. Therefore we may conclude.
\end{proof}

\begin{remark}
With the notation of Proposition \ref{betti}, notice that $\phi$ allows also to recover the \emph{graded} Betti numbers of $J(\D)$ from those of $J(\Ds)$. Provided that the partition of the $1$-skeleton of $\D$ is known, it is enough to consider the natural $\ZZ^p$-grading both on $R$ and on $S$. The $\ZZ^p$-grading on  $S$  is given by the $p$-partition. 
\end{remark}

We will conclude this section with a first application of Equation \eqref{H}. We will find a formula the $h$-vectors $\HV^{\D}$ where $\D$ is a $1$-dimensional matroid. By Theorem \ref{duality}, this is equivalent to describing the $h$-vectors of $\Bbbk[\D]$, where $\D$ is a matroid such that its Stanley-Reisner ideal has height $2$.

By Proposition \ref{1dim}, a 1-dimensional matroid $\D$  is actually a complete $p$-partite graph on $n$ vertices. For all $k=1,\ldots ,n-1$, let us set
\[\AL_k(\D)=|\{i\in \{1,\ldots ,p\}~:~|A_i|\geq k\}|-1.\]

\begin{proposition}\label{characterizationcodimension2}
Let $\D$ be a $1$-dimensional matroid on $[n]$. For all $k=0,\ldots ,n-2$, we have 
\[\displaystyle \HV^{\D}(k)=\sum_{i=1}^{n-k-1}\AL_i(\D).\]
\end{proposition} 

\begin{proof}
Let us choose a vertex $v\in A_p$. Clearly, the cover ideal of the link of $v$ is the principal ideal
\[\displaystyle J(\lk_{\D}v)=\biggl(\prod_{i\in [n]\setminus A_p}x_i\biggr).\]
In particular, we have 
\begin{equation*}
\HV^{\lk_{\D}v}(i)= \left\{
\begin{array}{ll}
1& \text{if } 0\le i<n-|A_p|, \\
0&\text{otherwise.}
\end{array}\right.
\end{equation*}
 The partition sets of the matroid $\D\setminus v$ are $A_1,A_2,\ldots ,A_{p-1},A_p\setminus \{v\}$, so we have
\begin{equation*}
\AL_k(\D\setminus v) = \left\{
\begin{array}{ll}
\AL_k(\D) & \text{if }k\neq |A_p|,\\
\AL_k(\D)-1& \text{if } k=|A_p|.
\end{array} \right.
\end{equation*}
\noindent By induction we have
\[\HV^{\D\setminus v}(k)=\sum_{i=1}^{n-k-2}\AL_i(\D\setminus v),\] 
for all $k=0,\ldots ,n-3$. On the other side,  by \eqref{H} we have
\[\HV^{\D}(k)=\HV^{\D\setminus v}(k-1)+\HV^{\lk_{\D}v}(k), \ \ \ \forall \ k=0,\ldots ,n-1.\]
Therefore,
\begin{equation*}
\HV^{\D}(k)=\left\{
\begin{array}{llll}
\sum_{i=1}^{n-k-1}\AL_i(\D)-1 + \HV^{\lk_{\D}v}(k)&=&\sum_{i=1}^{n-k-1}\AL_i(\D)&  \text{if } k\leq n -|A_p|-1,\\
\rule{0pt}{6ex} 
\sum_{i=1}^{n-k-1}\AL_i(\D) + \HV^{\lk_{\D}v}(k)&=&\sum_{i=1}^{n-k-1}\AL_i(\D)&\text{otherwise.}
\end{array}\right.
\end{equation*}
\end{proof}

\begin{corollary}\label{corheight2}
For a sequence $h=(1,2,h_3,\ldots ,h_s)$, the following are equivalent:
\begin{itemize}
\item[{\em (i)}] There is a matroid $\D$ such that $h$ is the $h$-vector of $\Bbbk[\D]$.
\item[{\em (ii)}] There is a matroid $\D$ such that $h$ is the $h$-vector of $S/J(\D)$.
\item[{\em (iii)}] $h$ is a pure $O$-sequence.
\item[{\em (iv)}] $h$ is the $h$-vector of a level algebra.
\item[{\em (v)}] $h_{i+1}\leq 2h_i+h_{i-1}$ for all $i=1,\ldots ,s$.
\end{itemize} 
\end{corollary}
\begin{proof}
The equivalence between (i) and (ii) follows by Theorem \ref{duality}, whereas (iii) is equivalent to (iv) by  the Hilbert-Burch theorem. The equivalence between (iv) and (v) was shown by Iarrobino in \cite{Ia}. 
As $S/J(\D)$ is level, and thus (ii) implies (iv), we just need to prove that (v) implies (ii)
and this follows easily from Proposition  \ref{characterizationcodimension2}.
\end{proof}

\begin{corollary}
If $\D$ is a $1$-dimensional matroid, then $\type(\D)=p-1$, where $\D$ is $p$-partite. 
\end{corollary}
\section{Stanley's  Conjecture}\label{sec3}

The main result of this section is Theorem \ref{thmdeltat}, in which we prove that Stanley's conjecture holds for certain matroids which we identify in a natural way. The first discussion of this section and Theorem \ref{stanleyforcomplete} are particular cases of the main result. They are the starting point of  the inductive procedure in the proof of Theorem \ref{thmdeltat}. For a better understanding of the construction which we present here, we will start with a closer  look at an already known case of Stanley's Conjecture \ref{conjd}, namely the codimension two case. In this first part we will concentrate on examples which hopefully provide the necessary intuition for the more technical proofs.\\

Consider a 1-dimensional matroid $\D$ on $[n]$, thus by Proposition \ref{1dim} it is a complete $p$-partite graph. Recall that we denote the partition sets of the graph by $A_i$ and for  $i=1,\ldots,p$ we have $a_i=|A_i|$. For simplicity, we assume for the moment that $a_1\le\ldots\le a_p$.
We will now present an inductive method to compute the $h$-vector of $\D$.

When we restrict to the first layer $A_1$, we obtain a $0$-dimensional matroid on $a_1$ vertices. It is clear that in this case $J(\D_1) = (x_1\cdots x_{a_1})$ so the $h$-vector of $\D_1$ is the vector of length $a_1$: $(1,1,\ldots,1)$.
Let $v_{2,1}$ be the first vertex of the parallel class $A_2$. This vertex will be a cone-point of the $1$-dimensional matroid $\D|_{A_1\cup \{v_{2,1}\}}$, so the $h$-vector will be the same as the one of $\D_1$. 
We will now use the recursive formula \eqref{H} to compute the $h$-vector of $\D_{1,2}$. By Lemma \ref{lemma2} we have 
\[\lk_{\D_{1,2}} v_{2,i} = A_1,\quad\forall \ v_{2,i} \in A_2.\]
So the $h$-vector of $\D|_{A_1\cup\{v_{2,1},v_{2,2}\}}$ is computed as follows:
\[\begin{array}{cccccc}
  &1&1&\ldots&1&1\\
1&1&1&\ldots&1&0\\
\hline
\rule{0pt}{2.5ex}1&2&2&\ldots&2&1
\end{array}\]
where the first row represents the $h$-vector of $\D|_{A_1\cup \{v_{2,1}\}}$, the second row represents the $h$-vector of the link, i.e. of $\D_1$. The last row is the $h$-vector of  $\D|_{A_1\cup\{v_{2,1},v_{2,2}\}}$. To compute the $h$-vector of $\D|_{A_1\cup\{v_{2,1},v_{2,2},v_{2,3}\}}$ we proceed in the same way.
All together we have to apply this procedure $a_2-1$ times. This can be done also directly in the following way:
\[\begin{array}{cccccccccc}
 & &  &  &  &1&\ldots&1&1&1\\
 & &  &  &1&1&\ldots&1&1&0\\
 & &  &1&1&1&\ldots&1&0&0\\
  &&&&&&&&&\\
  &  &  &  &  &  &\ldots&  &  &  \\
  &&&&&&&&&\\
1&1&1&\ldots&1&0&\ldots&0&0&0\\
\hline
\rule{0pt}{2.5ex}1&2&3&\ldots&a_1&a_1&\ldots&3&2&1\\
\end{array}\]

\vspace{2mm}

The  $h$-vector of $\D_{1,2,3}$ of is computed in a similar way. The only difference is that $v_{3,1}$ will no longer be a cone point. Thus the first row will be $h^{\D_{1,2}}$ and the number of shifted rows will be $a_3$.
 Repeating this procedure,  we can imagine that $h^\D$ is computed summing the columns of the  staircase in Figure \ref{fig1}.
 Notice that the last nonzero entry of $h^\D$  is $p-1$.

\begin{figure}[!h]
\setlength{\unitlength}{5mm}
\begin{picture}(22,18)
\put(0,2){\line(1,0){15.75}}
\put(1,3){\line(1,0){16.75}}

\put(3,5){\line(1,0){14.75}}\put(18.25,5){\line(1,0){2.75}}
\put(4,6){\line(1,0){13.75}}\put(18.25,6){\line(1,0){2.75}}
\put(5,7){\line(1,0){12.75}}\put(18.25,7){\line(1,0){2.75}}
\put(6,8){\line(1,0){11.75}}
\put(7,9){\line(1,0){10.75}}\put(18.25,9){\line(1,0){2.75}}

\put(8,10){\line(1,0){9.75}}\put(18.25,10){\line(1,0){2.75}}
\put(10,12){\line(1,0){7.75}}\put(18.25,12){\line(1,0){2.75}}

\put(12,14){\line(1,0){5.75}}\put(18.25,14){\line(1,0){2.75}}
\put(13,15){\line(1,0){4.75}}\put(18.25,15){\line(1,0){2.75}}
\put(14,16){\line(1,0){3.75}}

\put(18.75,1){\line(1,0){3.25}}
\put(18.75,8){\line(1,0){3.25}}
\put(18.75,11){\line(1,0){3.25}}
\put(18.75,16){\line(1,0){3.25}}
\put(0,1){\line(0,1){1}}
\put(1,1){\line(0,1){2}}
\put(2,1){\line(0,1){2.75}}
\put(3,1){\line(0,1){2.75}}\put(3,4.25){\line(0,1){0.75}}
\put(4,1){\line(0,1){2.75}}\put(4,4.25){\line(0,1){1.75}}
\put(5,1){\line(0,1){2.75}}\put(5,4.25){\line(0,1){2.75}}
\put(6,1){\line(0,1){2.75}}\put(6,4.25){\line(0,1){3.75}}
\put(7,1){\line(0,1){2.75}}\put(7,4.25){\line(0,1){4.75}}
\put(8,1){\line(0,1){2.75}}\put(8,4.25){\line(0,1){5}}\put(8,9.75){\line(0,1){0.25}}
\put(9,1){\line(0,1){2.75}}\put(9,4.25){\line(0,1){5}}\put(9,9.75){\line(0,1){1.25}}
\put(10,1){\line(0,1){2.75}}\put(10,4.25){\line(0,1){5}}\put(10,9.75){\line(0,1){2.25}}
\put(11,1){\line(0,1){2.75}}\put(11,4.25){\line(0,1){5}}\put(11,9.75){\line(0,1){3}}
\put(12,1){\line(0,1){2.75}}\put(12,4.25){\line(0,1){5}}\put(12,9.75){\line(0,1){3}}\put(12,13.25){\line(0,1){0.75}}
\put(13,1){\line(0,1){2.75}}\put(13,4.25){\line(0,1){5}}\put(13,9.75){\line(0,1){3}}\put(13,13.25){\line(0,1){1.75}}
\put(14,1){\line(0,1){2.75}}\put(14,4.25){\line(0,1){5}}\put(14,9.75){\line(0,1){3}}\put(14,13.25){\line(0,1){3.75}}
\put(15,1){\line(0,1){2.75}}\put(15,4.25){\line(0,1){5}}\put(15,9.75){\line(0,1){3}}\put(15,13.25){\line(0,1){2.75}}
\put(16,1){\line(0,1){2.75}}\put(16,4.25){\line(0,1){5}}\put(16,9.75){\line(0,1){3}}\put(16,13.25){\line(0,1){2.75}}
\put(17,1){\line(0,1){2.75}}\put(17,4.25){\line(0,1){5}}\put(17,9.75){\line(0,1){3}}\put(17,13.25){\line(0,1){2.75}}
\put(19,1){\line(0,1){2.75}}\put(19,4.25){\line(0,1){5}}\put(19,9.75){\line(0,1){3}}\put(19,13.25){\line(0,1){2.75}}
\put(20,1){\line(0,1){2.75}}\put(20,4.25){\line(0,1){5}}\put(20,9.75){\line(0,1){3}}\put(20,13.25){\line(0,1){2.75}}
\put(21,1){\line(0,1){2.75}}\put(21,4.25){\line(0,1){5}}\put(21,9.75){\line(0,1){3}}\put(21,13.25){\line(0,1){3.75}}
\put(21.5,4){\vector(0,-1){3}}
\put(21.5,5){\vector(0,1){3}}
\put(21.3,4.5){$a_p$}

\put(21.5,13){\vector(0,-1){2}}
\put(21.5,14){\vector(0,1){2}}
\put(21.3,13.5){$a_2$}

\put(17,16.5){\vector(-1,0){3}}
\put(18,16.5){\vector(1,0){3}}
\put(17.2,16.3){$a_1$}
\linethickness{0.5mm}
\put(0,1){\line(1,0){17.75}}\put(18.25,1){\line(1,0){2.75}}
\put(6,8){\line(1,0){11.75}}\put(18.25,8){\line(1,0){2.75}}
\put(9,11){\line(1,0){8.75}}\put(18.25,11){\line(1,0){2.75}}
\linethickness{5mm}
\put(20.5,1){\line(0,1){2.75}}\put(20.5,4.25){\line(0,1){2.75}}
\put(19.5,1){\line(0,1){2.75}}\put(19.5,4.25){\line(0,1){1.75}}
\put(16.5,1){\line(0,1){2}}
\put(15.5,1){\line(0,1){1}}

\put(19.5,8){\line(0,1){1.25}}
\put(20.5,8){\line(0,1){1.25}}\put(20.5,9.75){\line(0,1){0.25}}

\put(20.5,11){\line(0,1){1.75}}\put(20.5,13.25){\line(0,1){1.75}}
\put(19.5,11){\line(0,1){1.75}}\put(19.5,13.25){\line(0,1){0.75}}

\linethickness{3.75mm}
\put(17.375,1){\line(0,1){2.75}}
\put(18.625,1){\line(0,1){2.75}}
\put(18.625,4.25){\line(0,1){0.75}}
\put(17.375,11){\line(0,1){1}}
\put(18.625,11){\line(0,1){1.75}}
\put(18.625,8){\line(0,1){1}}
\multiput(0.25,1.25)(1,0){16}{$1$}
\multiput(1.25,2.25)(1,0){16}{$1$}

\multiput(4.25,5.25)(1,0){13}{$1$}
\multiput(5.25,6.25)(1,0){12}{$1$}\multiput(19.25,6.25)(1,0){1}{$1$}
\multiput(6.25,7.25)(1,0){11}{$1$}\multiput(19.25,7.25)(1,0){2}{$1$}
\multiput(7.25,8.25)(1,0){10}{$1$}

\multiput(9.25,10.25)(1,0){8}{$1$}\multiput(19.25,10.25)(1,0){2}{$1$}
\multiput(10.25,11.25)(1,0){7}{$1$}\multiput(19.25,11.25)(1,0){2}{$1$}

\multiput(13.25,14.25)(1,0){4}{$1$}\multiput(19.25,14.25)(1,0){2}{$1$}
\multiput(14.25,15.25)(1,0){3}{$1$}\multiput(19.25,15.25)(1,0){2}{$1$}

\multiput(2.35,3.75)(1,0){19}{$\cdot$}
\multiput(8.35,9.25)(1,0){13}{$\cdot$}
\multiput(11.35,12.75)(1,0){10}{$\cdot$}
\multiput(17.875,1.25)(0,1){15}{$\cdot$}
\end{picture}
\caption{Computing $h^\D$ for $d=2$}\label{fig1}
\end{figure}
In Figure \ref{scara2} we can see one example of how the corresponding order ideal is constructed in the case when $\D$ is the 1-dimensional matroid on 15 vertices, with 4-partition (3,3,4,5). Notice that the columns contain monomials of the same degree and that the exponent  of $x$ is constant on the rows.
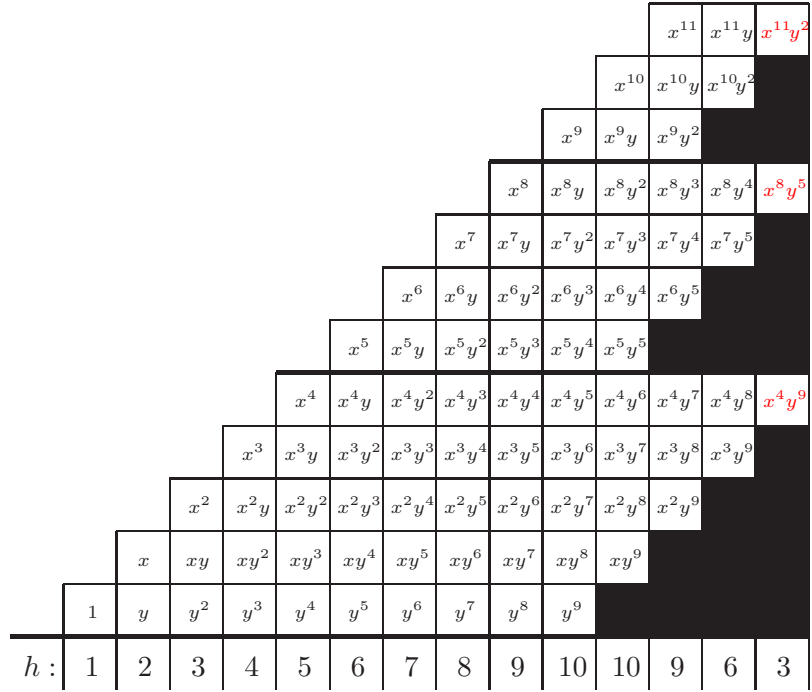
\begin{figure}[!h]
\setlength{\unitlength}{7mm}
\begin{picture}(15,13)
\put(1,2){\line(1,0){14}}
\put(2,3){\line(1,0){13}}
\put(3,4){\line(1,0){12}}
\put(4,5){\line(1,0){11}}

\put(6,7){\line(1,0){9}}
\put(7,8){\line(1,0){8}}
\put(8,9){\line(1,0){7}}

\put(10,11){\line(1,0){5}}
\put(11,12){\line(1,0){4}}
\put(12,13){\line(1,0){3}}

\put(1,0){\line(0,1){2}}
\put(2,0){\line(0,1){3}}
\put(3,0){\line(0,1){4}}
\put(4,0){\line(0,1){5}}
\put(5,0){\line(0,1){6}}
\put(6,0){\line(0,1){7}}
\put(7,0){\line(0,1){8}}
\put(8,0){\line(0,1){9}}
\put(9,0){\line(0,1){10}}
\put(10,0){\line(0,1){11}}
\put(11,0){\line(0,1){12}}
\put(12,0){\line(0,1){13}}
\put(13,0){\line(0,1){13}}
\put(14,0){\line(0,1){13}}
\put(15,0){\line(0,1){13}}

\linethickness{0.5mm}
\put(0,1){\line(1,0){15}}
\put(5,6){\line(1,0){10}}
\put(9,10){\line(1,0){6}}

\linethickness{7mm}
\put(14.5,1){\line(0,1){4}}
\put(13.5,1){\line(0,1){3}}
\put(12.5,1){\line(0,1){2}}
\put(11.5,1){\line(0,1){1}}

\put(14.5,6){\line(0,1){3}}
\put(13.5,6){\line(0,1){2}}
\put(12.5,6){\line(0,1){1}}

\put(14.5,10){\line(0,1){2}}
\put(13.5,10){\line(0,1){1}}

{\tiny
\put(1.45,1.35){$1$}
\put(2.4,1.35){$y$}
\put(3.35,1.35){$y^{2}$}
\put(4.35,1.35){$y^{3}$}
\put(5.35,1.35){$y^{4}$}
\put(6.35,1.35){$y^{5}$}
\put(7.35,1.35){$y^{6}$}
\put(8.35,1.35){$y^{7}$}
\put(9.35,1.35){$y^{8}$}
\put(10.35,1.35){$y^{9}$}

\put(2.4,2.35){$x$}
\put(3.3,2.35){$xy$}
\put(4.25,2.35){$xy^{2}$}
\put(5.25,2.35){$xy^{3}$}
\put(6.25,2.35){$xy^{4}$}
\put(7.25,2.35){$xy^{5}$}
\put(8.25,2.35){$xy^{6}$}
\put(9.25,2.35){$xy^{7}$}
\put(10.25,2.35){$xy^{8}$}
\put(11.25,2.35){$xy^{9}$}

\put(3.35,3.35){$x^{2}$}
\put(4.25,3.35){$x^{2}y$}
\put(5.15,3.35){$x^{2}y^{2}$}
\put(6.15,3.35){$x^{2}y^{3}$}
\put(7.15,3.35){$x^{2}y^{4}$}
\put(8.15,3.35){$x^{2}y^{5}$}
\put(9.15,3.35){$x^{2}y^{6}$}
\put(10.15,3.35){$x^{2}y^{7}$}
\put(11.15,3.35){$x^{2}y^{8}$}
\put(12.15,3.35){$x^{2}y^{9}$}

\put(4.35,4.35){$x^{3}$}
\put(5.15,4.35){$x^{3}y$}
\put(6.15,4.35){$x^{3}y^{2}$}
\put(7.15,4.35){$x^{3}y^{3}$}
\put(8.15,4.35){$x^{3}y^{4}$}
\put(9.15,4.35){$x^{3}y^{5}$}
\put(10.15,4.35){$x^{3}y^{6}$}
\put(11.15,4.35){$x^{3}y^{7}$}
\put(12.15,4.35){$x^{3}y^{8}$}
\put(13.15,4.35){$x^{3}y^{9}$}

\put(5.35,5.35){$x^{4}$}
\put(6.15,5.35){$x^{4}y$}
\put(7.15,5.35){$x^{4}y^{2}$}
\put(8.15,5.35){$x^{4}y^{3}$}
\put(9.15,5.35){$x^{4}y^{4}$}
\put(10.15,5.35){$x^{4}y^{5}$}
\put(11.15,5.35){$x^{4}y^{6}$}
\put(12.15,5.35){$x^{4}y^{7}$}
\put(13.15,5.35){$x^{4}y^{8}$}
\put(14.15,5.35){\color{red}$x^{4}y^{9}$}

\put(6.35,6.35){$x^{5}$}
\put(7.15,6.35){$x^{5}y$}
\put(8.15,6.35){$x^{5}y^{2}$}
\put(9.15,6.35){$x^{5}y^{3}$}
\put(10.15,6.35){$x^{5}y^{4}$}
\put(11.15,6.35){$x^{5}y^{5}$}

\put(7.35,7.35){$x^{6}$}
\put(8.15,7.35){$x^{6}y$}
\put(9.15,7.35){$x^{6}y^{2}$}
\put(10.15,7.35){$x^{6}y^{3}$}
\put(11.15,7.35){$x^{6}y^{4}$}
\put(12.15,7.35){$x^{6}y^{5}$}

\put(8.35,8.35){$x^{7}$}
\put(9.15,8.35){$x^{7}y$}
\put(10.15,8.35){$x^{7}y^{2}$}
\put(11.15,8.35){$x^{7}y^{3}$}
\put(12.15,8.35){$x^{7}y^{4}$}
\put(13.15,8.35){$x^{7}y^{5}$}

\put(9.35,9.35){$x^{8}$}
\put(10.15,9.35){$x^{8}y$}
\put(11.15,9.35){$x^{8}y^{2}$}
\put(12.15,9.35){$x^{8}y^{3}$}
\put(13.15,9.35){$x^{8}y^{4}$}
\put(14.15,9.35){\color{red}$x^{8}y^{5}$}

\put(10.35,10.35){$x^{9}$}
\put(11.15,10.35){$x^{9}y$}
\put(12.15,10.35){$x^{9}y^{2}$}

\put(11.35,11.35){$x^{10}$}
\put(12.15,11.35){$x^{10}y$}
\put(13.1,11.35){$x^{10}\!y^{2}$}

\put(12.35,12.35){$x^{11}$}
\put(13.15,12.35){$x^{11}y$}
\put(14.1,12.35){${\color{red}x^{11}\!y^{2}}$}
}

\put(0.25,0.25){$h:$}
\put(1.4,0.25){$1$}
\put(2.4,0.25){$2$}
\put(3.4,0.25){$3$}
\put(4.4,0.25){$4$}
\put(5.4,0.25){$5$}
\put(6.4,0.25){$6$}
\put(7.4,0.25){$7$}
\put(8.4,0.25){$8$}
\put(9.4,0.25){$9$}
\put(10.3,0.25){$10$}
\put(11.3,0.25){$10$}
\put(12.4,0.25){$9$}
\put(13.4,0.25){$6$}
\put(14.4,0.25){$3$}

\end{picture}
\caption{One order ideal which produces $h^\D$ the $4$-partition $(3,3,4,5)$}\label{scara2}
\end{figure}
Depending on the order of the parallel classes we can build a total of 12 different staircases, each one producing an order ideal. Eliminating the symmetry given by exchanging $x$ and $y$, we are left with 6 different order ideals with the right $f$-vector. For example ordering the partition as $(4,3,3,5)$ we obtain the order ideal generated by $\{x^4y^9,x^7y^6,x^{10}y^3\}$.

In higher dimensions the picture becomes more complicated. One can either imagine $d$-dimensional staircases, where each cube has value 1, or 2-dimensional staircases, where each row is the $h$-vector of the link of a parallel class. As we already saw, the order of the $a_i$'s plays no role in the computation of $h^\D$, providing us with several ways to construct an order ideal with the same $f$-vector. A complicated example in dimension 2, with 6-partite 1-skeleton shows that unfortunately with this method there is  no ``canonical'' choice. By canonical we understand a construction that should be independent of the values of the $a_i$'s.  \\

There is one case in which the choice of the order ideal is unique, namely the case when $d=p$. As we will see in Remark \ref{Gor=ci}, this is equivalent to $J(\D)$ being Gorenstein.
\begin{lemma}\label{d=p}
If $\D$ is a $(d-1)$-dimensional, $d$-partite matroid (so $d=p$) with partition $(a_1,\ldots, a_d)$, then 
\[h^\D = f(\langle y_1^{a_1-1}\cdots y_d^{a_d-1}\rangle).\]
\end{lemma}
\begin{proof}
The minimal generators of $J(\D)$ are the monomials corresponding to the basic covers of $\D$. In this situation, $A_1,\ldots ,A_d$ are the unique basic covers of $\D$, so $J(\D)$ is a complete intersection with $d$ generators of degrees $a_1,\ldots ,a_d$. The conclusion follows because the $h$-vector of a complete intersection depends only on the degree of its minimal generators. 
\end{proof}

We will now define  a class of matroids and prove that the Stanley conjecture holds for this class. When one fixes the dimension and the $p$-partition of the vertex set, these matroids will have all the admissible faces, thus they are in a sense a generalization of the Gorenstein matroids.  
\begin{definition}\label{def complete}
Let $\D$ be a $d-1$-dimensional matroid on $[n]$ with $p$-partite 1-skeleton. 
We say that $\D$ is a \emph{complete $p$-partite matroid}  if
\[A_{i_1}\ldots A_{i_d} \in \D,\quad \textup{for any subset~} \{i_1,\ldots,i_d\}\subset \{1,\ldots,p\}.\]
\end{definition}
Whenever $p$ is clear from the context, we will just call $\D$ \emph{complete}. Notice that a complete matroid is uniquely determined by the cardinalities of the parallel classes $a_1,\ldots,a_p$ and by $d$. 
It is  also clear that a matroid is complete if\mbox{}f  its simplification $\Ds$ is the \emph{uniform matroid} $U_{d,p}$ (see \cite[p. 17]{Ox11}). 
Complete matroids also generalize \emph{partition matroids} (see \cite[p. 18]{Ox11}), which correspond to the case $p=d$. 
In Proposition \ref{1dim} we proved that for $d=2$ all matroids are complete. For $d> 2$ this is no longer true, as the following easy example shows.
\begin{example}
Let $n=4$ and $\D = \{ \{1,2,3\},\{1,2,4\},\{1,3,4\}\}$. It is clear that $\D$ is a matroid. The 1-skeleton of $\D$ is
\[\D^1 = \{\{1,2\},\{1,3\},\{1,4\},\{2,3\},\{2,4\},\{3,4\}\} = K_4,\]
so it is a complete 4-partite graph. This means that  $a_1=a_2=a_3=a_4=1$. Clearly this matroid is not complete, as the face $\{2,3,4\}$ is missing. The complete 2-dimensional matroid corresponding to the above $a_i$'s is $\D' = \{ \{1,2,3\},\{1,2,4\},\{1,3,4\},\{2,3,4\}\}$.
\end{example}
\begin{remark}\label{complete link}
Let $\D$ be a complete $p$-partite matroid. We have
\begin{itemize}
\item[(i)] For any subset of vertices $M\subset [n]$ the restriction of $\D$ to $M$ is also a complete matroid.
\item[(ii)] For any parallel class $A_i$, the link in $\D$ of any of its vertices $\lk_\D A_i$ is also a complete matroid.
\end{itemize}
\end{remark}
\begin{theorem}\label{stanleyforcomplete}
Let $\D$ be a complete, $(d-1)$-dimensional matroid with $p$-partition of the 1-skeleton $A_1,\ldots,A_p$. For $i=1,\ldots,p$ we denote by $a_i=|A_i|$. Let $\G$ be the pure multi-complex on $\{y_1,\ldots,y_d\}$ with facets
\[ \F(\G)=\{ y_1^{(\sum_{i=l_0}^{l_1-1}a_i) -1}y_2^{(\sum_{i=l_1}^{l_2-1}a_i) -1}\cdots y_d^{(\sum_{i=l_{d-1}}^{p}a_i) -1}~:~ \forall~1=l_0<l_1<l_2<\ldots<l_{d-1}\le p\}.\]
Then we have that $$\HV^\D = f(\G),$$ where $\HV^\D$ is the $h$-vector of the algebra $S/J(\D)$.
\end{theorem}
 Before we start the proof, let us make a few easy remarks and  introduce some notation.
For each  $i \in \{d,\ldots,p\}$, we denote the link of the $i$-parallel class in the restriction of $\D$ to the first $i$ parallel classes by
\[ L_i= \lk_{\D_{1,\ldots,i}}A_i.\]
Notice that $L_i$ is the $(d-2)$-skeleton of $\D_{1,\ldots,i-1}$. We will write $r(i)$ for the length of the $h$-vector of $L_i$. 
As the number of vertices of $L_i$ is $a_1+\ldots+a_{i-1}$ and its dimension is $d-2$, we have that 
\[r(i) = 2-d +\sum_{j=1}^{i-1} a_j.\]

\begin{proof}
We will prove this theorem by simultaneous  induction on $d$ and $p-d$. The case $d=1$ is trivially true and by Lemma \ref{d=p} we know  that the theorem is true for $p=d$. 

For each $i$, denote by $\G^{L_i}$ the pure multi-complex corresponding to $L_i$ which is given by the inductive hypothesis. 
We assume now that $p>d>1$ and that $\HV^{\D_{1,\ldots,p-1}} =f(\G_{p-1})$, where 
\[ \G_{p-1}= \langle y_1^{(\sum_{i=1}^{l_1-1}a_i) -1}y_2^{(\sum_{i=l_1}^{l_2-1}a_i) -1}\cdots y_{d}^{(\sum_{i=l_{d-2}}^{p}a_i) -1}~:~ \forall~1<l_1<l_2<\ldots<l_{d-1}\le p-1\rangle.\]
We will use  $\HV^{\D_{1,\ldots,p-1}}$ and $\HV^{L_p}$ to compute $\HV^\D$ via the formula given in \eqref{H}. Clearly this formula has to be applied $a_p$ times, once for every vertex in $A_p$. So we obtain
\begin{equation}\label{inductcomplete}
\HV^\D(j) = \HV^{\D_{1,\ldots,p-1}}(j-a_p)+ \sum_{k=0}^{a_p-1}\HV^{L_p}(j-k),
\end{equation}
for all $0\le j \le 1-d+\sum_{k=1}^p a_k$.
To conclude we just need to check that the $f$-vectors of $\G$,  $\G_{p-1}$ and  $\G^{L_p}$ satisfy the same formula. To this purpose, for any $j\in \ZZ$, let us denote $F_j=\{ M\in \G~:~\deg M = j\}$, $G_j=\{ M\in \G^{L_p}~:~\deg M = j\}$ and $H_j=\{ M\in \G_{p-1}~:~\deg M = j\}$. Let us furthermore partition $F_j$ as
\[\displaystyle F_j=F_{j,\geq a_p}\bigcup \left(\bigcup_{k=0}^{a_p-1}F_{j,k}\right),\]
where $F_{j,\geq a_p}=\{M\in F_j~:~y_d^{a_p}\mid M\}$ and $F_{j,k}=\{M\in F_j~:~y_d^{k}\mid M \mbox{ and }y_d^{k+1}\nmid M\}$. It is easy to check the bijections of sets
\begin{eqnarray*}
G_{j-a_p} & \xrightarrow{\cong} & F_{j,\geq a_p} \\
M & \mapsto & M\cdot y_d^{a_p}
\end{eqnarray*}
and, for all $k=0,\ldots, a_p-1$,
\begin{eqnarray*}
H_{j-k} & \xrightarrow{\cong} & F_{j,k} \\
M & \mapsto & M\cdot y_d^{k}
\end{eqnarray*}
Therefore we get the formula
\[f_j(\G)=f_{j-a_p}(\G_{p-1})+\sum_{k=0}^{a_p-1}f_{j-k}(\G^{L_p}) \ \ \ \forall \ j\in \ZZ,\]
which, together with \eqref{inductcomplete}, yields the conclusion by induction. 
\end{proof}

Fixing two positive integers $d$ and $n$ and a vector $\aaa = (a_1,\ldots ,a_p)\in (\ZZ_+)^p$ such that $p\geq d$, $a_1+\ldots +a_p = n$, we introduce the class 
\[\M(d,p,\aaa),\] 
consisting of all $(d-1)$-dimensional matroids with  $p$-partite  $1$-skeleton, where the partition sets $A_i$ have cardinality $a_i$ for all $i=1,\ldots ,p$. Note that the classes $\M(d,p,\aaa)$ depend only on the set $\{a_1,\ldots ,a_p\}$. That is, $\M(d,p,\aaa)$ coincides with $\M(d,p,\aaa^{\sigma})$ for any permutation $\sigma$ of $p$ elements ($\aaa^{\sigma}$ means $(a_{\sigma(1)},\ldots ,a_{\sigma(p)})$). Furthermore notice that, if $d=2$ or $p=d$, $\M(d,p,\aaa)$ consists of a single matroid, but this happens only in these cases. To see this, it is enough to consider for $t=0,\ldots ,d-2$, the following simplicial complexes
\begin{equation}\label{deltat}
\D_t(d,p,\aaa)=\langle \{v_1,v_2,\ldots ,v_t,v_{i_1},\ldots ,v_{i_{d-t}}\} ~:~ t< i_1< \ldots <i_{d-t}\leq p \mbox{ where }v_i\in A_i \rangle.
\end{equation}
It is easy to see  that $\D_t(d,p,\aaa)$ are elements of $\M(d,p,\aaa)$. Moreover, one can show that, if $p>d$, they are not isomorphic pairwise - the easiest way to show this is to notice that they have a different number of facets. The matroid $\D_0(d,p,\aaa)$ is just the complete $p$-partite matroid whose partition sets $A_1,\ldots ,A_p$ satisfy $|A_i|=a_i$ for all $i=1,\ldots ,p$. Notice that, apart from the case  $t=0$, the matroid $\D_t(d,p,\aaa)$  depends on the vector $\aaa$, not just on the set of its entries. 

\begin{remark} For every $t, d, n$ and $\aaa$ as above the matroids $\D_t(d,p,\aaa)$ are representable. To see this it is enough to notice that their simplification satisfies
\[\Ds_t(d,p,\aaa) = \{v_{1},\ldots,v_{t}\}\ast U_{d-t,p-t} = \langle \{v_{1},\ldots,v_{t}\}\cup F~:~ F \in \F(U_{d-t,p-t})\rangle ,\]
where the $v_i$'s are fixed vertices and $U_{d-t,p-t}$ is the uniform matroid of rank $d-t$ on $p-t$ vertices.
Thus, a representation of  $\D_t(d,p,\aaa)$ is obtained by taking  $a_i$ copies of the $i$th column ($i=1,\ldots,p$)  in a representation of $\{v_{1},\ldots,v_{t}\}\ast U_{d-t,p-t}$.
Furthermore, it is easy to check that in order to obtain a representation over a field $\mathbb{F}$, its cardinality has to be ``large enough''.
\end{remark}
As a first thing, we want to show that Stanley's conjecture holds true for all $\D_t(d,p,\aaa)$.

\begin{theorem}\label{thmdeltat}
Let $d,p\in \NN$ be such that $p\geq d\geq 1$ and $\aaa = (a_1,\ldots ,a_p)\in (\ZZ_+)^p$. Then $h^{\D_t(d,p,\aaa)}$ is a pure {\it O}-sequence for all $t=0,\ldots ,d-2$.
\end{theorem}
\begin{proof}
The case $t=0$ has already been treated in Theorem \ref{stanleyforcomplete}. So, we will use induction on $t$, assuming that $t\geq 1$. Let us write $\D_t$ for $\D_t(d,p,\aaa)$. The restricted simplicial complex $\D_t'=(\D_t)_{2,3,\ldots ,p}$ is just $\D_{t-1}(d-1,p-1,\tilde{\aaa})$, where $\tilde{\aaa}=(a_2,\ldots ,a_p)$. Therefore, we know by induction that $h^{\D_t'}$ is a pure $O$-sequence.
Set $A_1=\{v_{1,1},\ldots ,v_{1,a_1}\}$ and $\D_t^i\subset \D_t$ the sub-complex induced by the vertices $A_2\cup \ldots \cup A_p\cup \{v_{1,1},\ldots ,v_{1,i}\}$  for all $i=1,\ldots ,a_1$. We have $h^{\D_t^1}=h^{\D_t'*a_{1,1}}=h^{\D_t'}$. Moreover, for all $i\geq 2$ and $k\in \ZZ$, we have
\[h^{\D_t^i}(k)=h^{\D_t^{i-1}}(k-1)+h^{\D_t'}(k).\]
Particularly, since $\D_t=\D_t^{a_1}$, we get
\begin{equation}\label{thmcones}
h^{\D_t}(k)=\sum_{j=0}^{a_1-1}h^{\D_t'}(k-j) \ \ \ \ \ \forall \ k\in \ZZ.
\end{equation}
We know that $h^{\D_t'}$ is a pure $O$-sequence, so let $\Gamma'$ be the order ideal such that $f_{\Gamma'}=h^{\D_t'}$. Let us suppose that the set of maximal degree monomials  of $\Gamma'$ is
\[\F_{\Gamma'}=\{u_1,\ldots ,u_s~:~u_i\in \Bbbk[y_2,\ldots ,y_{d}] \mbox{ and } \deg(u_i)=a_2+\ldots +a_p-d+1\}.\]
Let $\Gamma$ be the pure order ideal with the following set of maximal monomials:
\[\F(\Gamma)=\{u_1y_1^{a_1-1},\ldots ,u_sy_1^{a_1-1}\}.\]
One can easily see that
\[f_{\Gamma}(k)=\sum_{j=0}^{a_1-1}f_{\Gamma'}(k-j), \ \ \ \ \ \forall \ k \in \ZZ,\]
so \eqref{thmcones} yields the conclusion.
\end{proof}

Putting together Theorem \ref{stanleyforcomplete} and the proof of Theorem \ref{thmdeltat} we obtain an explicit construction for an order ideal with the $f$-vector we are looking for. Namely, we obtain the following corollary.
\begin{corollary} If we denote by $\Gamma_t(d,p,\aaa)$ the following order ideal:
\[ \langle y_1^{a_1-1}\cdots y_t^{a_t-1}y_{t+1}^{(\sum_{i=t+1}^{l_1-1}a_i) -1}\cdots y_d^{(\sum_{i=l_{d-t-1}}^{p}a_i) -1}~:~ \forall~t+1<l_1<l_2<\ldots<l_{d-t-1}\le p\rangle,\]
we have that 
\[h^{\D_t(d,p,\aaa)} = f(\Gamma_t(d,p,\aaa)).\]
In particular,
\begin{equation}\label{socleequation}
\displaystyle \type(S/J(\D_t(d,p,\aaa)))=\binom{p-t-1}{d-t-1}
\end{equation}
\end{corollary}

A consequence of Theorem \ref{thmdeltat} is the following interesting fact:

\begin{corollary}\label{d+1parallel classes}
Let $d\geq 1$. For all $\aaa \in (\ZZ_+)^{d+1}$ and $\D\in \M(d,d+1,\aaa)$, $h^{\D}$ is a pure {\it O}-sequence.
\end{corollary}
\begin{proof}
We want to show that $\D$ actually is $\D_t(d,p,\aaa)$ for some $t=0,\ldots ,d-2$, so that Theorem \ref{thmdeltat} would give the thesis. Passing to $\Ds$, a proof in the case $\aaa = \uno = (1,1,\ldots ,1)\in (\ZZ)_+^{d+1}$ is enough. Notice that any $(d-1)$-dimensional pure simplicial complex on the vertex set $\{1,\ldots ,d+1\}$ is a matroid. In order to have the complete graph on $d+1$ vertices as $1$-skeleton, $\Ds$ must have $m\geq 3$ facets. Moreover, if $\D$ is a $(d-1)$-simplicial complex on $d+1$ vertices with $m\geq 3$ facets, then it is easy to prove that $\D$ is isomorphic to the matroid $\D_{d-m+1}(d,d+1,\uno)$. 
\end{proof}

\section{Minimal and maximal $h$-vectors}\label{sec4}

Among the matroids described in \eqref{deltat},  two play a fundamental role:
\begin{eqnarray}\label{deltamaxmin}
\DM(d,p,\aaa)&=&\D_0(d,p,\aaa),\\
\Dm(d,p,\aaa)&=&\D_{d-2}(d,p,\aaa^{\sigma}),\nonumber
\end{eqnarray}
where $\sigma$ is a permutation of $p$ elements such that $a_{\sigma(1)}\leq \ldots \leq a_{\sigma(p)}$.
In this section we will see that, for any $\D\in \M(d,p,\aaa)$, we have
\[h^{\Dm(d,p,\aaa)}\leq h^{\D}\leq h^{\DM(d,p,\aaa)}\]
component-wise.

 Given a matroid $\D$ with parallel classes $A_1,\ldots ,A_p$,  we need to consider in the following lemma the matroid $\D_{r\leftrightarrow s}$, where the parallel classes $A_r$ and $A_s$ are switched. Let us give a more rigorous definition: The matroid  $\D_{r\leftrightarrow s}$ has as facets the subsets $F=\{v_{i_1},\ldots ,v_{i_d}\}$ of $[n]$ such that one of the following happens:
\begin{compactitem}
\item[(i)] $|F\cap (A_r\cup A_s)|\in\{0,2\}$ and $F\in\F(\D)$,
\item[(ii)] $v_{i_j}\in A_r$, $F\cap A_s =\emptyset$ and there exists $v\in A_s$ such that $(F\setminus\{v_{i_j}\})\cup \{v\}\in \F(\D)$,
\item[(iii)] $v_{i_k}\in A_s$, $F\cap A_r=\emptyset$ and there exists $u\in A_r$ such that $(F\setminus\{v_{i_k}\})\cup \{u\}\in \F(\D)$.
\end{compactitem}

\begin{lemma}\label{switch}
Let $p>d$ and $\aaa = (a_1,\ldots ,a_p)\in (\ZZ_+)^p$ be a vector such that $a_1\leq \ldots \leq a_p$. Let $\D\in \M(d,p,\aaa)$ be a matroid such that $A_p$ is a cone class for $\D$ (i.e. it corresponds to a cone point in $\Ds$). Pick $\ell\in\{1,\ldots ,p-1\}$ such that $A_\ell$ is not a cone class for $\D$ (it exists because $p>d$). Then
\[h^{\D_{\ell\leftrightarrow p}}\leq h^{\D}.\]
\end{lemma}
\begin{proof}
Set $L_p=\lk_{\D}A_p$ and $\overline{L}_\ell=\lk_{\D_{\ell\leftrightarrow p}}A_{\ell}$. Furthermore, let $L_p'=L_p\setminus A_\ell$  and $\overline{L'}_\ell=\overline{L}_{\ell} \setminus A_p$. Notice that $L_p'\cong \overline{L'}_\ell$ and  that $T=\lk_{L_p}A_\ell \cong \lk_{\overline{L'}_\ell}A_p=U$. For all $k\in\ZZ$, we have
\begin{equation}\label{eqswitch1}
h^\D(k) = \sum_{i=0}^{a_p-1}h^{L_p'}({k-a_\ell-i})+\sum_{i=0}^{a_p-1}\sum_{j=1}^{a_\ell}h^{T}({k-a_{\ell}-i+j})
\end{equation}
and 
\begin{equation}\label{eqswitch2}
h^{\D_{\ell\leftrightarrow p}}(k) = \sum_{i=0}^{a_\ell-1}h^{\overline{L'}_\ell}({k-a_p-i})+\sum_{i=0}^{a_\ell-1}\sum_{j=1}^{a_p}h^{U}(_{k-a_p-i+j}).
\end{equation}
From the above discussion we  have $h^{L_p'}(r)=h^{\overline{L'}_\ell}(r)=:h_r'$ and $h^T(r)=h^U(r)=:h_r''$ for all $r\in \ZZ$. Let us set 
\[M_1=\sum_{i=0}^{a_p-1}h_{k-a_\ell-i}', \ \ \ M_2=\sum_{i=0}^{a_p-1}\sum_{j=1}^{a_\ell}h_{k-a_{\ell}-i+j}''\] 
and
\[N_1=\sum_{i=0}^{a_\ell-1}h_{k-a_p-i}', \ \ \ N_2=\sum_{i=0}^{a_\ell-1}\sum_{j=1}^{a_p}h_{k-a_p-i+j}''.\]
Because $a_\ell\leq a_p$, obviously $N_1\leq M_1$. Moreover we claim that $N_2=M_2$. To see this, it is enough to notice that 
\[h_{k-a_p -i+j}''=h_{k-a_\ell-(a_p-j)+(a_\ell-i)}''.\]
So, we get that \eqref{eqswitch2} is less than or equal to \eqref{eqswitch1}.
\end{proof}

We need one more technical lemma.

\begin{lemma}\label{lemmamin}
Let $A_1,\ldots ,A_p$ and $B_1,\ldots ,B_q$ be partitions of $\{1,\ldots ,n\}$ of cardinality $|A_i|=a_i$ and $|B_j|=b_j$, where $p\geq d$ and $q\geq d$, such that
\begin{enumerate}
\item[{\em (i)}] $a_1\leq \ldots \leq a_p$,
\item[{\em (ii)}] $b_1\leq \ldots \leq b_q$,
\item[{\em (iii)}] $B_j=\bigcup_{k=1}^{r_j}A_{i_{j,k}}$,
\item[{\em (iv)}] $\bigcup_{i=1}^dA_i\subset \bigcup_{i=1}^dB_i$.
\end{enumerate}
Set $\tilde{\aaa}=(a_1,a_2,\ldots ,a_{d-1},a_d+\ldots +a_p)$ and $\bbb=(b_1,\ldots ,b_q)$. If $\Gamma$ is the only $(d-1)$-dimensional matroid in $\M(d,d,\tilde{\aaa})$, then
\[h^{\Gamma}\leq h^{\D} \ \ \ \forall \ \D\in \M(d,q,\bbb).\]
\end{lemma}
\begin{proof}
If $q=d$, then the assertion can be deduced by inspection on the $h$-vectors of $\D$ and $\Gamma$, described in Theorem \ref{stanleyforcomplete}. In fact, using this theorem, one can show a more general statement: Let ${\bm \alpha}=(\alpha_1,\ldots ,\alpha_d)\in (\ZZ_+)^d$ and ${\bm \beta}=(\beta_1,\ldots ,\beta_d)\in (\ZZ_+)^d$ be vectors such that $\alpha_1\leq \ldots \leq \alpha_d$, $\beta_1\leq \ldots \leq \beta_d$, $\sum_{i=1}^d\alpha_i=\sum_{i=1}^d\beta_i$ and $\alpha_i\leq \beta_i$ for all $i=1,\ldots ,d-1$. Then, the $h$-vector of the only matroid in $\M(d,d,{\bm \alpha})$ is less than or equal to the $h$-vector of the only matroid in $\M(d,d,{\bm \beta})$. We leave the easy proof of this fact to the reader.

We will use  induction on $p$.  Notice that, as we always have $d\leq q\leq p$, 
the case $p=d$, implies $q=d$, so we are done by the above discussion.

If $p>d$ and $q>d$, then $i_{q,k}>d$ for all $k=1,\ldots ,r_q$. Consider the sub-complex $\Gamma'\subset \Gamma$ induced by the vertices not in $B_q$ and set $L=\lk_\Gamma B_q$.  As $B_q$ is a subset of a parallel class in $\Gamma$, L is well defined and for all $k\in \ZZ$ we have
\[h^{\Gamma}(k)=h^{\Gamma'}({k-b_q})+\sum_{i=1}^{b_q}h^{L}({k-b_q+i}).\]

In the same vein, we can consider the sub-complex $\Dp\subset \D$ induced by all the vertices of $\D$ not in $B_q$ and we set $K=\lk_{\D}B_q$. Once again we have, for all $k\in \ZZ$, 
\[h^{\D}(k)=h^{\D'}({k-b_q})+\sum_{i=1}^{b_q}h^{K}({k-b_q+i}).\]
By Lemma \ref{switch} we can assume that $B_q$ is not a cone class of  $\D$, so that $\Dp$  has dimension $d-1$. Therefore by the induction on $p$ we immediately get $h^{\Gamma'}\leq h^{\D'}$. 

On the other hand, $L$ is the unique $(d-1)$-partite $(d-2)$-dimensional matroid on the partition $(a_1,\ldots ,a_{d-1})$, whereas $K$ is a $(d-2)$-dimensional matroid on a certain partition $C_1,\ldots ,C_r$. For sure $r\geq d-1$ and, provided that $|C_1|\leq \ldots \leq |C_r|$, we get also that $a_i\leq b_i\leq |C_i|$ for all $i=1,\ldots d-1$. Take a facet $\{v_{i_1},\ldots ,v_{i_{d-1}}\}$ of $K$ and suppose that each $v_{i_k}\in C_{i_k}$. Then the sub-complex $K'\subset K$ induced by the vertices of $C_{i_1}\cup \ldots \cup C_{i_{d-1}}$ is a complete $(d-1)$-partite $(d-2)$-dimensional matroid. We can assume $ i_1<\ldots <i_d$, so that $a_k\leq |C_{i_k}|$ for all $k=1,\ldots ,d-1$. So we can choose $a_k$ vertices in each one of the $C_{i_k}$s. It turns out that $L$ is isomorphic to the sub-complex of $K'$ induced by these vertices. Therefore $L$ is isomorphic to an induced sub-complex of $K$, which implies $h^{L}\leq h^{K}$. So we can conclude.
\end{proof}

\begin{theorem}\label{minimality}
If $d\geq 1$ and $\aaa=(a_1,\ldots ,a_p)\in (\ZZ_+)^p$ with $p\geq d$, then
\[h^{\Dm(d,p,\aaa)}\leq h^{\D} \ \ \ \forall \ \D\in \M(d,p,\aaa).\]
\end{theorem}
\begin{proof}
Since neither the matroid $\Dm(d,p,\aaa)$ nor the set $\M(d,p,\aaa)$ depend on the order of $a_1,\ldots ,a_p$, we are allowed to assume that $a_1\leq \ldots \leq a_p$.
We  induct on $p$. If $p=d$, then the theorem is trivial, since $\M(d,d,\aaa)$ consists of only one matroid, namely $\Dm(d,d,\aaa)$. If $p>d$, let us set $\Dm(d,p,\aaa)'\subset \Dm(d,p,\aaa)$ and $\D'\subset \D$ the sub-complexes induced by the vertices of $A_1\cup \ldots \cup A_{p-1}$. Furthermore set $\aaa'=(a_1,\ldots ,a_{p-1})$. We have that $\Dm(d,p,\aaa)'=\Dm(d,p-1,\aaa')$. Exploiting Lemma \ref{switch}, we can assume that $\dim \D' = d-1$, so that $\D'\in \M(d,p-1,\aaa')$. So, by induction, we get 
\[h^{\Dm(d,p,\aaa)'}\leq h^{\D'}.\]
Now set $L=\lk_{\Dm(d,p,\aaa)}A_p$ and $K=\lk_{\D}A_p$. It turns out that $L$ is the unique $(d-1)$-partite $(d-2)$-dimensional matroid on the partition $(a_1,\ldots ,a_{d-2},a_{d-1}+\ldots +a_{p-1})$. Instead $K$ will be a $(d-2)$-dimensional matroid on a certain partition $(b_1,\ldots ,b_q)$. Such partitions satisfy the hypotheses of Lemma \ref{lemmamin}, so we get 
\[h^L\leq h^K.\]
This yields the conclusion, since for all $k\in \ZZ$
\begin{eqnarray*}
h^{\Dm(d,p,\aaa)}(k)&=&h^{\Dm(d,p,\aaa)'}({k-a_p})+\sum_{i=1}^{a_p}h^L({k-a_p+i})\ \ \ \textrm{and} \\
h^{\D}(k)&=&h^{\D'}({k-a_p})+\sum_{i=1}^{a_p}h^K({k-a_p+i}).
\end{eqnarray*}
\end{proof}

\begin{remark}\label{Gor=ci}
By Theorem \ref{minimality} and \eqref{socleequation} one has that, for all $\D\in \M(d,p,\aaa)$,
\begin{equation}\label{socleinequality}
\type(S/J(\D))\geq p-d+1.
\end{equation}
This implies that, for any matroid $\D$, $S/I_{\D}$ is Gorenstein if and only if $I_{\D}$ is a complete intersection if and only if $p=d$. So we recover \cite[Theorem 4.4.10]{Sto}.
\end{remark}

Equation \eqref{socleinequality} allows us to prove the following:

\begin{theorem}
If $\D$ is a matroid on $\{1,\ldots ,n\}$ such that $\type(S/I_{\D})\leq 3$, then $h(\D)=h(\Bbbk[\D])$ is a pure {\it O}-sequence. Equivalently, if the $h$-vector of a matroid has the shape $(1,h_1,\ldots ,h_s)$ with $h_s\leq 3$, then it is a pure $O$-sequence.
\end{theorem}
\begin{proof}
First of all we replace $I_{\D}$ for $J(\D)$. If $\type(S/J(\D))\leq 2$, then $\D$ has to belong or to $\M(d,d,\aaa)$ or to $\M(d,d+1,\aaa)$ thanks to Equation \eqref{socleinequality}, so in these cases the statement follows at once by Lemma \ref{d=p} and Corollary \ref{d+1parallel classes}. Thus we have only to care of the case $\type(S/J(\D))=3$. Again using Equation \eqref{socleinequality}, Lemma \ref{d=p} and Corollary \ref{d+1parallel classes}, we can assume that $\D$ is in $\M(d,d+2,\aaa)$. The Cohen-Macaulay type of $S/J(\D)$ is the same as the one of $S/J(\Ds)$ by Proposition \ref{betti}. But the dual of $\Ds$ is a rank 2 matroid (possibly on less than $d+2$ vertices), so it is some complete $p$-partite graph $G$ (see Proposition \ref{1dim}). Furthermore $h^{\Ds}=h_G=(1,h_1,h_2)$, where $h_2=e-v+1$ ($v$ denotes the number of vertices of $G$ and $e$ the number of its edges). But we want $h_2=3$, that is $e=v+2$. It is easy to check that the only complete $p$-partite graph on $v$ vertices with $v+2$ edges is the complete graph on $4$ vertices. This means that $\Ds=\D_{d-2}(d,d+2,{\bf 1})$, so $\D=\D_{d-2}(d,d+2,\aaa)$. Now the conclusion follows from Theorem \ref{thmdeltat}.  
\end{proof}

To show that $\DM(d,p,\aaa)$ has maximal $h$-vector among the matroids $\D\in \M(d,p,\aaa)$ is much easier.

\begin{theorem}\label{maximality}
If $d\geq 1$ and $\aaa=(a_1,\ldots ,a_p)\in (\ZZ_+)^p$ with $p\geq d$, then
\[h^{\D}\leq h^{\DM(d,p,\aaa)} \ \ \ \forall \ \D\in \M(d,p,\aaa).\]
\end{theorem}
\begin{proof}
It is harmless to assume that $\Bbbk$ is infinite; otherwise we can tensor with its algebraic closure. Looking at the respective  vertex covers, it is clear that $J(\DM(d,p,\aaa))\subset J(\D)$ for all $\D\in\M(d,p,\aaa)$. Since both $S/J(\DM(d,p,\aaa))$ and $S/J(\D)$ are $(n-d)$-dimensional Cohen-Macaulay  rings, we can choose $n-d$ linear forms which are both $S/J(\DM(d,p,\aaa))$- and $S/J(\D)$- regular (the generic ones have this property). Passing  to the Artinian reduction,  the inclusion is preserved, so we infer the desired inequality.
\end{proof}

\end{document}